\documentclass[11pt]{article}
\usepackage[british]{babel}
\usepackage{amsmath}
\usepackage{amssymb}
\usepackage{amsfonts}
\usepackage{amsthm}
\usepackage{graphicx}

\newtheorem{theorem}{Theorem}

\newtheorem{proposition}[theorem]{Proposition}
\newtheorem{definition}[theorem]{Definition}

\makeatletter
   \renewcommand{\section}{\@startsection{section}{1}{0mm}
   {\baselineskip}%
   {\baselineskip}{\normalfont\normalsize\scshape\centering}}%
   \makeatother

\allowdisplaybreaks[4]
\begin{document}

\begin{center}
	\textsc{\Large On the Entropy of a Family of Random Substitutions} \vspace{1.5ex}
	
	\textsc{Johan Nilsson}\vspace{1.5ex}
	
	{\small\texttt{jnilsson@math.uni-bielefeld.de}}
	
	{\small\today}
\end{center}

\begin{abstract}
The generalised random Fibonacci chain is a stochastic extension of the classical Fibonacci substitution and is defined as the rule mapping $0\mapsto 1$ and $1 \mapsto 1^i01^{m-i}$ with probability $p_i$, where $p_i\geq 0$ with $\sum_{i=0}^m p_i =1$, and where the random rule is applied each time it acts on a $1$. We show that the topological entropy of this object is given by the growth rate of the set of inflated generalised random Fibonacci words.
\end{abstract}

{\small\noindent 2010 Mathematics Subject Classification: 68R15 Combinatorics on words, 05A16 Asymptotic enumeration, 37B10 Symbolic dynamics.}

\section{Introduction}

Meyer sets form an important class of point sets that can be considered as generalisations of lattices \cite{meyer,moody1}. They have recently been studied in connection with mathematical models of quasicrystals, see \cite{moody2,lagarias,baake11} and references therein. One of their common features is the existence of a non-trivial point spectrum. It shows up as a relatively dense set of point measures in their diffraction measure \cite{strungaru}, even under the constraint that one only considers point measures whose intensity is at least a given positive fraction of the the central intensity.  

The majority of papers so far has concentrated on deterministic Meyer sets, such as those obtained from Pisot inflations or from the projection method. All these examples have zero entropy. On the other hand, it is well known that Meyer sets with entropy exist (such as the set $2\mathbb{Z}$ combined with an arbitrary subset of $2\mathbb{Z}+1$: compare \cite{baake}) but relatively little is known about them. Despite having entropy, the spectral result of Strungaru \cite{strungaru} still applies, and the point spectrum is non-trivial. 

One possibility to define special classes of Meyer sets with entropy is via random inflation rules, as proposed in \cite{godreche}. Here we follow their line of approach for a family of generalised Fibonacci substitutions, parametrised by a number $m\in\mathbb{N}$. The spectral nature for $m=1$ was studied in \cite{godreche,claudia}. Here we reconsider this case and its generalisation from the point of view of entropy. In fact we prove that the topological entropy can simply be calculated by a suitable (and easily accessible) subset of sub-words (or factors), thus proving an (implicit) conjecture from \cite{godreche} and its generalisation to the entire family.   

Let us introduce the generalised random Fibonacci chain by the generalised substitution
\begin{equation}
\label{eq: def of Theta}
	\theta : 
	\left\{
	\begin{array}{rcl}
		0 & \mapsto & 1 \\
		1 & \mapsto &
		\left\{
		\begin{array}{ll}
			01^m & \textnormal{with probability $p_0$} \\
			101^{m-1} & \textnormal{with probability $p_1$} \\
			\vdots \\
			1^m0 & \textnormal{with probability $p_m$}
		\end{array}
		\right.
	\end{array}
	\right.
\end{equation}
where $p_i\geq0$ and $\sum_{i=0}^{m}p_i=1$ and where the random rule is applied each time $\theta$ acts on a $1$. 

In \cite{godreche} Godr\`eche and Luck define the random Fibonacci chain by the generalised substitution given by $\theta$, from (\ref{eq: def of Theta}), in the special case $m=1$. They introduce the random Fibonacci chain when studying quasi-crystal\-line structures and tilings in the plane. In their paper it is claimed without proof that the topological entropy of the random Fibonacci chain is given by the growth rate of the set of inflated random Fibonacci words. In \cite{nilsson} this fact was proven, and we shall here give a proof of a more general result when we consider the generalised random Fibonacci substitution given by $\theta$ from (\ref{eq: def of Theta}). 

Before we can state our main theorem, we need to introduce some notation. A word $w$ over an alphabet $\Sigma$ is a finite sequence $w_1w_2\ldots w_n$ of symbols from $\Sigma$. We let here $\Sigma = \{0,1\}$. We denote a sub-word or a factor of $w$ by $w[a,b] = w_a w_{a+1}w_{a+2}\ldots w_{b-1}w_b$ and similarly for a set of words we let $W[a,b] =\{w[a,b]: w\in W\}$. By $|\cdot|$ we mean the length of a word and the cardinality of a set. Note that $|w[a,b]| = b-a+1$.

For two words $u = u_1u_2u_3\ldots u_n$ and $v = v_1v_2v_3\ldots v_m$ we denote by $uv$ the concatenation of the two words, that is, $uv = u_1u_2u_3\ldots u_n v_1 v_2 \ldots v_m$. Similarly we let for two sets of words $U$ and $V$ their product be the set $UV = \{uv: u\in U, v\in V\}$ containing all possible concatenations.

Letting $\theta$ act on the word $0$ repeatedly yields an infinite sequence of words $r_n = \theta^{n-1}(0)$. We know that $r_1=0$ and $r_2=1$. But $r_3$ is one of the words $1^i01^{m-i}$ for $ 0\leq i\leq m$ with probability $p_i$ The sequence $\{r_n\}_{n=1}^{\infty}$ converges in distribution to an infinite random word $r$. We say that $r_n$ is an inflated word (under $\theta$) in generation $n$ and we introduce here sets that correspond to all inflated words in generation $n$;

\begin{definition}
\label{def: recursive def An}
For a fixed $m\geq 1$ let $A_0 := \emptyset$, $A_{1} := \{ 0 \}$ and $A_{2} := \{ 1 \}$ and for $n\geq3$  define recursively 
\begin{equation}
\label{eq: rec def An}
	A_n = \bigcup_{i=0}^{m}\prod_{j=0}^{m}A_{n-1-\delta_{ij}},
\end{equation}
where $\delta$ is the Kronecker symbol, that is $\delta_{ij}=1$ if $i=j$ and 0 otherwise and we let $A := \lim_{n\to\infty}A_n$. 
\end{definition}

We shall later on, as a direct consequence of Proposition \ref{prop: An prefix}, see that the set $A$ is a well-defined set. The integer $m\geq1$ is arbitrary, but fixed during the whole process. Therefore we choose to not have $m$ as a parameter when denoting a set of inflated generalised random Fibonacci words, $A_n$.

It is clear from the definition of $A_n$ that all elements in $A_n$ have the same length, that is, for all $x,y \in A_n$ we have $|x| = |y|$. We shall frequently use the length of the elements in $A_n$, and therefore we introduce the notation 
\[	l_n := |x| \quad \textnormal{for}\quad x\in A_n,\quad n\geq1.
\]
From the recursion (\ref{eq: rec def An}) and the definition of $l_n$ we have immediately the following proposition

\begin{proposition}
\label{prop: l_n rec}
The numbers $l_n$ fulfil the recursion relation, $l_1 =1$,  $l_2 = 1$ and $l_n = m\cdot l_{n-1} + l_{n-2}$ for $n\geq 3$.
\end{proposition}

For a word $w$ we say that $x$ is a \emph{factor} of $w$ if there are two words $u,v$ such that $w= uxv$. The factor set $F(S,n)$, of a set of words $S$, is the set of all factors of length $n$ of the words in $S$.  
We introduce the abbreviated notation $F_n$ for the special factor sets of the generalised random Fibonacci words,
\[	F_n := F(A,l_n).
\]
It seems to be a hard problem to give an explicit expression for the size of $|F_n|$. We have to leave that question open, but we shall give a rough upper bound of $|F_n|$. 

The topological entropy of the generalised random Fibonacci chain is de\-fined as the limit $\lim_{n\to\infty} \frac{1}{n}\log |F(A,n)|$. The existence of this limit is direct by Fekete's lemma \cite{fekete} since we have sub-additivity, $\log|F(A,n+p)|\leq\log|F(A,n)|+\log|F(A,p)|$.
We can now state the main result in this paper.

\begin{theorem}
\label{thm: lim An = lim Fn}
The logarithm of the growth rate of the size of the set of inflated generalised random Fibonacci words equals the topological entropy of the generalised random Fibonacci sequence, that is 
\begin{equation}
\label{eq: lim An = lim Fn}
	\lim_{n\to\infty}\frac{\log|A_n|}{l_n} 
	= \lim_{n\to\infty}\frac{\log|F_n|}{l_n}.
\end{equation}
\end{theorem}

The outline of the paper is that we start by studying the set $A_n$ by looking at sets of prefixes of the set $A_n$. Next we give a finite method for finding the factor set $F_n$ and finally we present an estimate of $|F_n|$ in terms of $|A_n|$, leading to the proof of Theorem \ref{thm: lim An = lim Fn}.

\section{Sets of Inflated Words}

In this section we shall study the sets of inflated generalised Fibonacci words, the $A_n$ sets. The main results of the section give how to calculate the size of $A_n$. The underlying idea in proving our results here is to rewrite and consider the union (\ref{eq: rec def An}) as a union of two sets only. We obtain this by collecting  all but the last union term in (\ref{eq: rec def An}) into one set and similarly all but the first one into a second set. We can then factor out an $A_n$-term from these sets, and thereby apply an induction argument on the achieved smaller set. 

The first result of this section gives how to simplify the overlap we obtain when rewriting and generalising  (\ref{eq: rec def An}) in the manner just described above.

\begin{proposition}
\label{prop: overlap}
For a fixed $m\geq1$ we have for $k\geq 1$ and $n\geq 2$
\begin{multline}
	 \left(A_{n} \Big(\bigcup_{i=0}^{k}\prod_{j=0}^{k} A_{n-\delta_{ij}}\Big)\right)
		\bigcap
		\left(\Big( \bigcup_{i=0}^{k}\prod_{j=0}^{k} A_{n-\delta_{ij}}\Big) A_{n}\right) = \\
	= A_{n}\Big(\bigcup_{i=0}^{k-1}\prod_{j=0}^{k-1} A_{n-\delta_{ij}}\Big) A_{n}.
\label{eq: overlap}
\end{multline}
\end{proposition}

\begin{proof}
It is clear that 
\[	A_{n} \Big(\bigcup_{i=0}^{k}\prod_{j=0}^{k} A_{n-\delta_{ij}}\Big)
	\supseteq
	A_{n} \Big(\bigcup_{i=0}^{k-1}\prod_{j=0}^{k} A_{n-\delta_{ij}}\Big)
	=
	A_{n} \Big(\bigcup_{i=0}^{k-1}\prod_{j=0}^{k-1} A_{n-\delta_{ij}}\Big)A_{n}
\]
and symmetrically 
\[	\Big(\bigcup_{i=0}^{k}\prod_{j=0}^{k} A_{n-\delta_{ij}}\Big)A_{n}  
	\supseteq
	A_{n} \Big(\bigcup_{i=0}^{k-1}\prod_{j=0}^{k-1} A_{n-\delta_{ij}}\Big)A_{n}.
\]
This gives the inclusion 
\begin{multline}
	 \left(A_{n} \Big(\bigcup_{i=0}^{k}\prod_{j=0}^{k} A_{n-\delta_{ij}}\Big)\right)
		\bigcap
		\left(\Big( \bigcup_{i=0}^{k}\prod_{j=0}^{k} A_{n-\delta_{ij}}\Big) A_{n}\right) \supseteq\\
	\supseteq A_{n}\Big(\bigcup_{i=0}^{k-1}\prod_{j=0}^{k-1} A_{n-\delta_{ij}}\Big) A_{n}.
	\label{eq: overlap supset}
\end{multline}
To prove that we actually have equality in (\ref{eq: overlap supset}); assume for contradiction that the inclusion in (\ref{eq: overlap supset}) is strict. Then there is an element $x$ in the left hand side intersection which is not in the right hand side set of (\ref{eq: overlap supset}). The word $x$ is then an element in 
\begin{equation}
	\left(A_{n} \Big(\bigcup_{i=0}^{k}\prod_{j=0}^{k} A_{n-\delta_{ij}}\Big)\right)
	\setminus \left(A_{n}\Big(\bigcup_{i=0}^{k-1}\prod_{j=0}^{k-1} A_{n-\delta_{ij}}\Big) A_{n}\right)
	\label{eq: overlap left set}
\end{equation}
and then also an element in the set
\begin{equation}
	\left(\Big(\bigcup_{i=0}^{k}\prod_{j=0}^{k} A_{n-\delta_{ij}}\Big)A_{n}\right)
	\setminus \left(A_{n}\Big(\bigcup_{i=0}^{k-1}\prod_{j=0}^{k-1} A_{n-\delta_{ij}}\Big) A_{n}\right).
	\label{eq: overlap right set}
\end{equation}
From (\ref{eq: overlap left set}) we see that $x$ must be an element in 
\[	\big(A_{n}^{k+1}A_{n-1}\big) \setminus \left(A_{n}\Big(\bigcup_{i=0}^{k-1}\prod_{j=0}^{k-1} A_{n-\delta_{ij}}\Big) A_{n}\right).
\]
This gives that $x$ must end with a word in $(A_{n}A_{n-1})\setminus \big((A_{n}A_{n})[1+l_{n}-l_{n-1}, 2l_n]\big)$. But then $x$ cannot be an element of (\ref{eq: overlap right set}), since all its elements ends with a word in $\big((A_{n}A_{n})[1+l_{n}-l_{n-1}, 2l_n]\big)$. A contradiction, hence must we have equality in (\ref{eq: overlap supset}). 
\end{proof}

Let us turn to proving a recursion for the size of the $A_n$ sets. The special case $k=m=1$ of (\ref{eq: An short rec}) was already considered by Godr\`eche and Luck in \cite{godreche}, see also Nilsson \cite{nilsson}.

\begin{proposition}
\label{prop: product k rec}
For a fixed $m\geq1$ we have for $k\geq 1$ and $n\geq 2$
\begin{equation}
\label{eq: product k rec}
	\left|\bigcup_{i=0}^{k}\prod_{j=0}^{k} A_{n-\delta_{ij}} \right|= \frac{m(n-2)+k+1}{m(n-2)+1}|A_{n}|^k|A_{n-1}|
\end{equation}
and in particular we have the special case $k=m$,
\begin{equation}
\label{eq: An short rec}
	|A_{n+1}|=\frac{m(n-1)+1}{m(n-2)+1}|A_{n}|^m|A_{n-1}|.
\end{equation}
\end{proposition}

\begin{proof}
We give a proof by induction on $n$. For the basis case $n=2$ we have
\[	\left| \bigcup_{i=0}^{k}\prod_{j=0}^{k} A_{2-\delta_{ij}} \right|  
	= k+1 
	= \frac{m(2-2)+k+1}{m(2-2)+1} |A_{2}||A_{1}| 
\]
for $1 \leq k$. Now assume for induction that (\ref{eq: product k rec}) holds for $ 2 \leq n\leq p $. Then for the induction step $n=p+1$ we give a proof by induction on $k$. For the basis case $k=1$ we have by the recursive definition of the $A_n$ sets, the induction assumption and by (\ref{eq: An short rec})
\begin{align*}
	\left| \bigcup_{i=0}^{1}\prod_{j=0}^{1} A_{p+1-\delta_{ij}} \right| 
	&= \left| A_p A_{p+1} \cup A_{p+1}A_{p} \right|\\
	&= \left| \Bigg(A_p \Big(\bigcup_{i=0}^{m}\prod_{j=0}^{m} A_{p-\delta_{ij}} \Big)\Bigg)
		\bigcup
		\Bigg(\Big(\bigcup_{i=0}^{m}\prod_{j=0}^{m} A_{p-\delta_{ij}} \Big) A_p\Bigg) \right| \\
	&= \left| \bigcup_{i=0}^{m+1}\prod_{j=0}^{m+1} A_{p-\delta_{ij}} \right|\\
	&= \frac{m(p-2)+m+2}{m(p-2)+1} |A_{p}|^{m+1}|A_{p-1}|\\
	&= \frac{m(p-2)+m+2}{m(p-2)+1} \cdot \frac{m(p-2)+1}{m(p-1)+1}|A_{p+1}||A_{p}|\\
	&= \frac{m(p-1)+1+1}{m(p-1)+1} |A_{p+1}||A_{p}|,
\end{align*}
which completes the first basis step. For the second basis step $k=2$ we similarly have
\begin{align*}
	&\left| \bigcup_{i=0}^{2}\prod_{j=0}^{2} A_{p+1-\delta_{ij}} \right| = \\
	&= \left| \Bigg(A_{p+1} \Big(\bigcup_{i=0}^{1}\prod_{j=0}^{1} A_{p+1-\delta_{ij}} \Big)\Bigg)
		\bigcup
		\Bigg(\Big(\bigcup_{i=0}^{1}\prod_{j=0}^{1} A_{p+1-\delta_{ij}} \Big) A_{p+1}\Bigg) \right| \\
	&= 2|A_{p+1}| \left| \bigcup_{i=0}^{1}\prod_{j=0}^{1} A_{p+1-\delta_{ij}} \right| 
		- |A_{p+1}|^{2}|A_{p}| \\
	&= 2\cdot\frac{m(p-1)+1+1}{m(p-1)+1} |A_{p+1}|^{2}|A_{p}| - |A_{p+1}|^{2}|A_{p}|\\
	&= \left(2\cdot\frac{m(p-1)+1+1}{m(p-1)+1}-1\right) |A_{p+1}|^{2}|A_{p}| \\
	&= \frac{m(p-1)+2+1}{m(p-1)+1} |A_{p+1}|^2|A_{p}|,
\end{align*}
which completes the second basis step on $k$. Now assume for induction that (\ref{eq: product k rec}) holds for $1\leq k \leq q$. Then we have for the induction step, $k=q+1$ by using Proposition \ref{prop: overlap} and the induction assumption
\begin{align*}
	&\left| \bigcup_{i=0}^{q+1}  \prod_{j=0}^{q+1} A_{p+1-\delta_{ij}} \right| = \\
	&= \left| \Bigg(A_{p+1} \Big(\bigcup_{i=0}^{q}\prod_{j=0}^{q} A_{p+1-\delta_{ij}}\Big)\Bigg)
		\bigcup
		\Bigg(\Big( \bigcup_{i=0}^{q}\prod_{j=0}^{q} A_{p+1-\delta_{ij}}\Big) A_{p+1}\Bigg) \right|\\
	&= 2|A_{p+1}| \left| \bigcup_{i=0}^{q}\prod_{j=0}^{q} A_{p+1-\delta_{ij}} \right| 
		- |A_{p+1}|^{2}  \left| \bigcup_{i=0}^{q-1}\prod_{j=0}^{q-1} A_{p+1-\delta_{ij}} \right| \\
	&= 2|A_{p+1}|\frac{m(p-1)+q+1}{m(p-1)+1}|A_{p+1}|^{q}|A_{p}| \\
		& \hspace{40mm}- |A_{p+1}|^{2} \frac{m(p-1)+q}{m(p-1)+1} |A_{p+1}|^{q-1}|A_{p}| \\
	&= \left(2\cdot \frac{m(p-1)+q+1}{m(p-1)+1} -\frac{m(p-1)+q}{m(p-1)+1}\right) |A_{p+1}|^{q+1}|A_{p}|\\
	&= \frac{m(p-1)+q+1+1}{m(p-1)+1}|A_{p+1}|^{q+1}|A_{p}|,
\end{align*}
which completes the induction on $k$, and thereby also the induction on $n$. 
\end{proof}

We can now give the explicit way of calculating the size of the set $A_n$, by simply unwinding the recursion given in (\ref{eq: An short rec}).

\begin{proposition}
For a fixed $m\geq 1$ we have for $n\geq 3$
\begin{equation}
	\label{eq: An explicit}
	|A_n| = \prod_{i=2}^{n-1} \big(m(n-i)+1\big)^{d_{i-1}}
\end{equation}
where $d_i = m\cdot d_{i-1} + d_{i-2}$ with $d_1 = 1$ and $d_2=m-1$.
\end{proposition}

\begin{proof}
We give a proof by induction on $n$. For the basis case $n=3$ we have by the recursion (\ref{eq: An short rec})
\[	|A_3| = \frac{m(3-2)+1}{m(3-3)+1}|A_2|^m |A_1| = m(3-2)+1,
\]
and similarly for the basis case $n=4$
\begin{align*}
	|A_4| 
	&= \frac{m(4-2)+1}{m(4-3)+1}|A_3|^m |A_2| \\
	&= \frac{m(4-2)+1}{m(4-3)+1} (m(3-2)+1)^m \\
	&= (m(4-2)+1)(m(4-3)+1)^{m-1}.
\end{align*}
Now assume that (\ref{eq: An explicit}) holds for $3\leq n \leq p$. Then we have for the induction step, $n=p+1$, by (\ref{eq: An short rec}) and the induction assumption
\begin{align*}
	|&A_{p+1}| = \\ 
	&= \frac{m(p-1)+1}{m(p-2)+1}|A_p|^m |A_{p-1}|\\
	&= \frac{m(p-1)+1}{m(p-2)+1}
	\left(  \prod_{i=2}^{p-1} \big(m(p-i)+1\big)^{d_{i-1}} \right)^m 
	\left(  \prod_{i=2}^{p-2} \big(m(p-1-i)+1\big)^{d_{i-1}} \right) \\
	&= \frac{m(p-1)+1}{m(p-2)+1}
	\left(  \prod_{i=2}^{p-1} \big(m(p-i)+1\big)^{d_{i-1}} \right)^m 
	\left(  \prod_{i=3}^{p-1} \big(m(p-i)+1\big)^{d_{i-2}} \right) \\
	&= (m(p-1)+1)(m(p-2)+1)^{m-1}
	\left(  \prod_{i=3}^{p-1} \big(m(p-i)+1\big)^{m\cdot d_{i-1} + d_{i-2}} \right) \\
	&= (m(p-1)+1)^{d_1}(m(p-2)+1)^{d_2}
	\left(  \prod_{i=3}^{p-1} \big(m(p-i)+1\big)^{d_{i}} \right) \\
	&= (m(p-1)+1)^{d_1}(m(p-2)+1)^{d_2}
	\left(  \prod_{i=4}^{p} \big(m(p+1-i)+1\big)^{d_{i-1}} \right) \\
	&= \prod_{i=2}^{(p+1)-1} \big(m(p+1-i)+1\big)^{d_{i-1}},
\end{align*}
which completes the induction.
\end{proof}

\section{Prefix Sets}

In this section we study some basic structures of the sets of prefixes of words from $A_n$.

\begin{proposition} For a fixed $m\geq1$ and any $n\geq 3$ and $k\geq0$ we have
\label{prop: An prefix}
\begin{equation}
	\label{eq: An prefix}
	A_n[1,l_n-1] = A_{n+k}[1,l_n-1].
\end{equation}
\end{proposition}

\begin{proof}
It is by definition clear that we have the inclusion $A_n[1,l_n-1] \subseteq A_{n+k}[1,l_n-1]$ for $n\geq 3$ and $k\geq0$. For the reversed inclusion we give a proof by induction on $k$ and on $n$. Let us first, in the basis case $k=1$, give two identities that we shall make use of later on in the proof. We have
\begin{align}
  A_{n+1}[1,l_n-1] 
	&= \left(\bigcup_{i=0}^{m}\prod_{j=0}^{m}A_{n-\delta_{ij}}\right)[1,l_n-1] \nonumber\\
	&= \left( (A_{n-1}A_n^m ) \bigcup \Big(\bigcup_{i=1}^{m}\prod_{j=0}^{m}A_{n-\delta_{ij}}\Big) \right)[1,l_n-1]\nonumber\\
	&= \big((A_{n-1}A_n)[1,l_n-1]\big) \bigcup \big(A_{n}[1,l_n-1]\big),  
	\label{eq: An prefix: first identity}
\end{align}
and secondly 
\begin{align}
 A_nA_{n+1} 
	&= A_{n} \bigcup_{i=0}^{m}\prod_{j=0}^{m}A_{n-\delta_{ij}} \nonumber\\
	&= \bigcup_{i=0}^{m}\prod_{j=-1}^{m}A_{n-\delta_{ij}} \nonumber\\
	&= \Big(\bigcup_{i=0}^{m-1}\prod_{j=-1}^{m}A_{n-\delta_{ij}}\Big)
	    \bigcup (A_n^{m+1}A_{n-1}) \nonumber\\
	&= \left( \Big(\bigcup_{i=0}^{m-1}\prod_{j=-1}^{m-1}A_{n-\delta_{ij}}\Big)A_n\right)
	    \bigcup (A_n^{m+1}A_{n-1}) \nonumber\\
	&\subseteq \left( \Big(\bigcup_{i=-1}^{m-1}\prod_{j=-1}^{m-1}A_{n-\delta_{ij}}\Big)A_n\right)
	    \bigcup (A_n^{m+1}A_{n-1}) \nonumber\\
	&= (A_{n+1}A_n) \bigcup (A_n^{m+1}A_{n-1}).
	\label{eq: An prefix: second identity}
\end{align}
Now let us turn to the first basis case $n=3$ in the induction on $n$. By the identity (\ref{eq: An prefix: first identity}) we have
\begin{align*}
  A_4[1,l_3-1] 
	&= \big((A_2A_3)[1,l_3-1]\big) \bigcup \big(A_3[1,l_3-1]\big)  \\
	&= \big((\{1\} A_3)[1,l_3-1] \big) \bigcup \big(A_{3}[1,l_3-1]\big)  \\
	&\subseteq \big( A_3[1,l_3-1]\big) \bigcup \big(A_{3}[1,l_3-1]\big)  \\
	&= A_{3}[1,l_3-1], 
\end{align*}
which concludes the case $n=3$. For the second basis case $n=4$ we have by (\ref{eq: An prefix: first identity})
\begin{align}
  A_5[1,l_4-1] 
	&= \big((A_3A_4)[1,l_4-1]\big) \bigcup \big(A_4[1,l_4-1]\big).
	\label{eq: An prefix: basis n=4}
\end{align}
The next step is to simplify the right hand side of (\ref{eq: An prefix: basis n=4}). By (\ref{eq: An prefix: second identity}) we have 
\begin{align*}
  (A_{3}A_{4})[1,l_{4}-1] 
	&\subseteq \big(( A_4 A_3) \bigcup (A_3^{m+1}A_2)\big)[1,l_4-1] \\
	&= \big( A_4[1,l_4-1] \big) \bigcup \big((A_3^{m+1}A_2)[1,l_4-1] \big)\\
	&= \big( A_4[1,l_4-1] \big) \bigcup A_3^m\\
	&=  A_4[1,l_4-1],
\end{align*}
which, combined by (\ref{eq: An prefix: basis n=4}) concludes the case. Now assume for induction on $n$ that (\ref{eq: An prefix}) holds for $k=1$ and $3 \leq n \leq p$. For the induction step $n=p+1$, we have by (\ref{eq: An prefix: first identity})
\begin{equation}
	A_{p+2}[1,l_{p+1}-1] = 
	\big((A_{p}A_{p+1})[1,l_{p+1}-1] \big) \bigcup \big(A_{p+1}[1,l_{p+1}-1] \big).
	\label{eq: union An prefix}
\end{equation}
As above, we aim to simplify (\ref{eq: union An prefix}). The identity (\ref{eq: An prefix: second identity}) and the induction assumption gives  
\begin{align}
  (A_{p}A_{p+1})&[1,l_{p+1}-1] \subseteq \nonumber\\
	&\subseteq \big(( A_{p+1}A_p) \bigcup (A_p^{m+1}A_{p-1})\big)[1,l_{p+1}-1] \nonumber\\
	&= \big(A_{p+1}[1,l_{p+1}-1]\big) \bigcup \big( A_p^{m+1} [1,l_{p+1}-1] \big) \nonumber\\
	&= \big(A_{p+1}[1,l_{p+1}-1]\big) \bigcup \Big( A_p^{m} \big(A_p[1,l_{p-1}-1] \big)\Big) \nonumber\\
	&= \big(A_{p+1}[1,l_{p+1}-1]\big) \bigcup \Big( A_p^{m} \big(A_{p-1}[1,l_{p-1}-1] \big)\Big) \nonumber\\
	&= \big(A_{p+1}[1,l_{p+1}-1]\big) \bigcup \big( (A_p^{m}A_{p-1})[1,l_{p+1}-1] \big) \nonumber\\
	&= A_{p+1}[1,l_{p+1}-1].
      \label{eq: An supset}
\end{align}
Combining (\ref{eq: union An prefix}) with (\ref{eq: An supset}) now gives 
\[	A_{p+2}[1,l_{p+1}-1] = A_{p+1}[1,l_{p+1}-1],
\]
which concludes the induction on $n$ and the basis case $k=1$ for the induction on $k$. Now assume for induction on $k$ that (\ref{eq: An prefix}) holds for $1\leq k \leq q$ and $n\geq 3$. Then we have for the induction step, $k=q+1$,  by using the induction assumption twice
\begin{align*}
	A_{n+q+1}[1,l_{n}-1]
	&= \big(A_{n+q+1}[1,l_{n+1}-1]\big)[1,l_{n}-1] \\
	&= \big(A_{(n+1)+q}[1,l_{n+1}-1]\big)[1,l_{n}-1] \\
	&= \big(A_{n+1}[1,l_{n+1}-1]\big)[1,l_{n}-1] \\
	&= A_{n+1}[1,l_{n}-1] \\
	&= A_{n}[1,l_{n}-1],
\end{align*}
which concludes the induction and the proof.
\end{proof}

The symmetry in the recursive definition of the $A_n$ sets (\ref{def: recursive def An}) gives that there clearly is a symmetric analogue of Proposition \ref{prop: An prefix}, that is the result holds for sets of suffixes. Next let us construct, by the help of \mbox{Proposition \ref{prop: An prefix}}, a superset to $A_n$ avoiding all the unions in the definition (\ref{def: recursive def An}).

\begin{proposition}
\label{prop: sup An}
Let 
\begin{align*}
	P_{n-1} &= \big(A_{n-1}[1,l_{n-1}-1]\big)\{0,1\},\\
	C_{n-1} &= \{0,1\}\big(A_{n-1}[2,l_{n-1}-1]\big)\{0,1\},\\
	S_{n-2} &= \{0,1\}\big(A_{n-2}[2,l_{n-2}]\big).
\end{align*}
Then for $n \geq 4 $ we have the inclusion
\[
	A_n \subseteq P_{n-1} C_{n-1}^{m-1} S_{n-2}.
\]
\end{proposition}

\begin{proof}
Let $B_n = P_{n-1} C_{n-1}^{m-1} S_{n-2}$ and let $T_i = A_{n-1}^iA_{n-2}A_{n-1}^{m-i}$ for $0 \leq i\leq m$. We have to show that $T_i\subseteq B_n$ for all $0 \leq i \leq m$. The case $i=m$ is direct by the construction, therefore we may assume $i<m$. By Proposition \ref{prop: An prefix} we have  
\begin{equation}
	T_i[1,(i+1) l_{n-1}] \subseteq A_{n}[1,(i+1) l_{n-1}] \subseteq B_{n}[1,(i+1) l_{n-1}].
	\label{eq: sup An first}
\end{equation}
Furthermore, we have from Proposition \ref{prop: An prefix}
\begin{align}
	A_{n-1}[1+l_{n-1}-l_{n-2}, l_{n-1}] 
	&\subseteq \{0,1\} \big(A_{n-1}[2+l_{n-1}-l_{n-2}, l_{n-1}]\big) \nonumber\\
	&=  \{0,1\} \big(A_{n-2}[2, l_{n-2}] \big) \label{eq: sup An second}\\
	&\subseteq  \{0,1\} \big(A_{n-1}[2, l_{n-2}] \big). \label{eq: sup An third}
\end{align}
Now combining (\ref{eq: sup An first}) and (\ref{eq: sup An third}) gives 
\[	T_i[1,(i+1) l_{n-1}+l_{n-2}] \subseteq B_{n}[1,(i+1) l_{n-1}+l_{n-2}].
\]
and in particular we have 
\begin{multline}
	T_i[1+i \cdot l_{n-1}+l_{n-2},(i+1) l_{n-1}+l_{n-2}] \subseteq \\
	\subseteq B_{n}[1+i \cdot l_{n-1}+l_{n-2},(i+1) l_{n-1}+l_{n-2}].
	\label{eq: sup An forth}
\end{multline}
Applying the inclusion (\ref{eq: sup An forth}) $m-1$ times and combining it the last time with (\ref{eq: sup An second}) concludes the proof.
\end{proof}

\section{Factor Sets}

The aim of this section is to give a finite method for finding the factor set $F_n$. 
In \cite{nilsson} we find a finite method for finding the factor set of $A_n$ in the case $m=1$.

\begin{proposition}
\label{prop: FAn+1 = FAnAn}
For a fixed $m\geq 1$ and all $n\geq3$ and $k\geq 1$ we have 
\begin{equation}
	\label{eq: FAn+1 = FAnAn}
	F(A_{n+k},l_n) = F(A_{n}^2,l_n).
\end{equation}
\end{proposition}

\begin{proof}
We give a proof by induction on $k$. For the basis case $k=1$ we look first at the special case $m=1$. It is clear that $F(A_{n+1},l_n)\subseteq F(A_{n}^2,l_n)$. For the reversed inclusion we have for 
$1\leq i < l_{n-1}$ by Proposition \ref{prop: An prefix},
\[	A^2_{n}[i,i-1+l_n] = (A_nA_{n-1})[i,i-1+l_n]\subseteq A_{n+1}[i,i-1+l_n].
\]
The remaining cases of $i$ follow by symmetry. 

For the case $m\geq 2$, it is clear that $F(A_{n}^2,l_n)\subseteq F(A_{n+1},l_n)$. For the reversed inclusion, we have by the recursive definition of the $A_n$ sets that a sub-word $w$ of length $l_n$ of a word in $A_{n+1}$ is an element of the set 
\begin{equation}
\label{eq: AnAnUAnAn-1U..}
	F(A_{n}^2,l_n)
	\bigcup F(A_{n-1}A_{n},l_n) 
	\bigcup F(A_{n}A_{n-1},l_n)
	\bigcup  F(A_{n}A_{n-1}A_n,l_n). 
\end{equation}
For all but the right-most set in (\ref{eq: AnAnUAnAn-1U..}) it is clear that they are subsets of $F(A^2_n,l_n)$. To deal with the remaining case assume $l_{n-1}< i \leq \lceil\frac{1}{2} (l_{n}+l_{n-1})\rceil$. 
This implies that 
\begin{equation}
	\label{eq: i length bound}
	i-1 + l_n \leq l_n + m\cdot l_{n-1},
\end{equation}
as $n\geq3$ and $m\geq2$. By the recursive definition of the $A_n$ sets (\ref{def: recursive def An}) and the bound (\ref{eq: i length bound}) on $i$ we obtain
\begin{align*}
	\big(A_{n}A_{n-1}&A_n\big)[i,i-1+l_n] = \\
	&= \left(A_{n}A_{n-1} \bigcup_{i=0}^{m}\prod_{j=0}^{m} A_{n-1-\delta_{ij}}  \right)[i,i-1+l_n] \\
	&= \left(A_{n} \bigcup_{i=0}^{m}\prod_{j=-1}^{m} A_{n-1-\delta_{ij}}  \right)[i,i-1+l_n] \\
	&= \left( \Big(A_{n} \bigcup_{i=0}^{m-1}\prod_{j=-1}^{m} A_{n-1-\delta_{ij}}\Big)  
	      \bigcup (A_nA_{n-1}^{m+1}A_{n-2}) \right)[i,i-1+l_n] \\
	&\subseteq \left( \Big(A_{n} \bigcup_{i=-1}^{m-1}\prod_{j=-1}^{m} A_{n-1-\delta_{ij}}\Big)  
	      \bigcup (A_nA_{n-1}^{m+1}A_{n-2}) \right)[i,i-1+l_n] \\
	&= \left( (A_{n}A_{n}A _{n-1})
	      \bigcup (A_nA_{n-1}^{m+1}A_{n-2}) \right)[i,i-1+l_n] \\
	&=  (A_{n}A_{n})[i,i-1+l_n]
	      \bigcup (A_nA_{n-1}^{m+1}) [i,i-1+l_n] \\
	&= (A_{n}A_{n})[i,i-1+l_n]  \bigcup (A_nA_{n-1}^{m}A_{n-2}) [i,i-1+l_n] \\
	&= \big(A_nA_n\big)[i,i-1+l_n] 
\end{align*}
Therefore we see that (\ref{eq: AnAnUAnAn-1U..}) is a subset to $F(A_{n}^2,l_n)$ when $l_{n-1}< i \leq \lceil\frac{1}{2} (l_{n}+l_{n-1})\rceil$. The remaining cases of $i$ follows by symmetry.

Now assume for induction that (\ref{eq: FAn+1 = FAnAn}) holds for $1\leq  k\leq q$. Then for $k= q+1$ we have by using the basis case and by the induction assumption
\begin{align*}
	F(A_{n+q+1},l_n) 
	&= F( F(A_{n+q+1},l_{n+1}), l_n)\\
	&= F( F(A_{n+1+1},l_{n+1}), l_n)\\
	&= F(A_{n+2}, l_n)\\
	&= F(A_{n+1}, l_n),
\end{align*}
which completes the induction.
\end{proof}

\begin{proposition}
\label{prop: Fn = FAA}
For a fixed $m\geq 1$ and all $n\geq 3$ we have 
\begin{equation}
	F_n = F(A_{n}^2,l_n).
\end{equation}
\end{proposition}

\begin{proof}
It is clear that $F(A_{n+1},l_n) \subseteq F_n$. For the reverse inclusion let $x\in F_n$. Then there are words $u$ and $v$ such that $uxv\in A_{n+k}$ for some integer $k\geq1$. That is $x\in F(A_{n+k},l_n)$. Proposition 
\ref{prop: FAn+1 = FAnAn} now gives $F(A_{n+k},l_n) = F(A_{n+1},l_n) = F(A_{n}^2,l_n)$ completing the proof.
\end{proof}

\section{Upper bound}

The aim of this section is to give an upper bound of the size of $|F_n|$ in terms of $|A_n|$.

\begin{proposition}
\label{prop: A[k]A[k]leq A}
For a fixed $m\geq1$ and for any $n\geq 3$ we have for all $1\leq k \leq l_n-1$
\begin{equation}
	\big|A_n[1,k]\big| \cdot \big|A_n[k+1,l_n]\big| \leq  4^{mn}|A_n|.
	\label{eq: An[k]An[k] leq An}
\end{equation}
\end{proposition}

\begin{proof} By symmetry we only need to consider the case with $1\leq k \leq \lfloor \frac{1}{2} l_n \rfloor$. We give a proof by induction on $n$. For the basis case $n=3$ we have
\[	\big|A_3[1,k]\big| \cdot \big|A_3[k+1,l_3]\big| \leq |A_3|^2 = (m+1)^2 \leq 4^{3m}(m+1) = 4^{3m}|A_3|.
\]
Now assume for induction that (\ref{eq: An[k]An[k] leq An}) holds for $3\leq n\leq p$. For the induction step $n=p+1$ let $0\leq i \leq m-1$ be such that 
\[	i\cdot l_p \leq k < (i+1) \cdot l_p.
\]
If $i=0$ we have by Proposition \ref{prop: sup An}, the induction assumption and (\ref{eq: An short rec}), when reusing the notation from Proposition \ref{prop: sup An},
\begin{align*}
	\big|A_{p+1}&[1,k]\big| \cdot \big|A_{p+1}[k+1,l_{p+1}]\big|\leq \\
	&\leq \big|(P_pC_{p}^{m-1} S_{p-1})[1,k]\big| \cdot \big|(P_pC_{p}^{m-1} S_{p-1})[k+1,l_{p+1}]\big| \\
	&= \big|P_{p}[1, k]\big| \cdot \big|P_p[k+1,l_{p}]\big| |C_p|^{m-1} |S_{p-1}| \\
	&= \big|A_p[1,k]\big| \cdot \big|A_p[k+1, l_p]\big| |A_p|^{m-1} 4^m |A_{p-1}| \\
	& \leq 4^{mp}|A_p| 4^m|A_p|^{m-1}|A_{p-1}| \\
	&= 4^{m(p+1)} |A_p|^{m}|A_{p-1}| \\
	&\leq 4^{m(p+1)} |A_{p+1}|,
\end{align*}
completing the induction in this case. Similarly, for the case $i\geq 1$ we have
\begin{align*}
	\big|A_{p+1}&[1,k]\big| \cdot \big|A_{p+1}[k+1,l_{p+1}]\big| \leq \\
	&\leq \big|(P_pC_{p}^{m-1} S_{p-1})[1,k]\big| \cdot \big|(P_pC_{p}^{m-1} S_{p-1})[k+1,l_{p+1}]\big| \\
	&= \big|(P_pC_{p}^{i})[1, k]\big| \cdot \big|C_{p}^{m-i}[k+1-i\cdot l_p,(m-i)l_{p}]\big| |S_{p-1}| \\
	&= 4^{i} |A_p|^{i} \big|A_p[1,k-i\cdot l_p]\big|
			\cdot \big|A_p[k+1-i\cdot l_p, l_p]\big| |A_p|^{m-i-1} 4^{m-i} |A_{p-1}| \\
	&= 4^{m} |A_p|^{m-1} \big|A_p[1,k-i\cdot l_p]\big|
			\cdot \big|A_p[k+1-i\cdot l_p, l_p]\big| |A_{p-1}|\\
	& \leq 4^{m} |A_p|^{m-1} 4^{mp}|A_p| |A_{p-1}|  \\
	&=  4^{m(p+1)} |A_p|^{m}|A_{p-1}| \\
	&\leq 4^{m(p+1)} |A_{p+1}|,
\end{align*}
which completes the case and the proof.
\end{proof}

The estimate in Proposition \ref{prop: A[k]A[k]leq A} is probably far from optimal. Computer calculations lead us to conjecture that we may replace the factor $4^{mn}$ in (\ref{eq: An[k]An[k] leq An}) by a constant. 

We can now give the estimate we set out to find. Again compared to computer calculations our estimate seems to be far from the optimal one, but it will be sufficient for our purpose.

\begin{proposition}
\label{prop: F_n leq 4An}
For a fixed $m\geq 1$  we have for $n\geq3$ 
\[	|F_n| \leq 4^{mn}\cdot l_{n}\cdot |A_n|.
\]
\end{proposition}

\begin{proof}
By Proposition \ref{prop: Fn = FAA} it is enough to estimate the size of $F(A_n^2,l_n)$. Let $x$ be a subword of length $l_n$ in $A_n^2$. That is, there is a $0\leq k < l_n$ such that 
\[	x\in A_n^2[k+1,k+l_n] = \big(A_n[k+1,l_n]\big)\big(A_n[1,k]\big).
\]
By Proposition \ref{prop: A[k]A[k]leq A} there are at most $4^{mn}|A_n|$ such words $x$. Adding up for all possible values of $k$ gives the desired result.
\end{proof}

\section{Proof of the Main Result}

Finally, we have now gathered enough background to prove our theorem.

\begin{proof}[Proof of Theorem \ref{thm: lim An = lim Fn}]
It is already clear that the limit $\lim_{n\to\infty} \frac{\log|F_n|}{l_n}$ exists. By Proposition \ref{prop: F_n leq 4An} we have $|A_n| \leq |F_n| \leq 4^{mn}l_n |A_n|$ and by Proposition \ref{prop: l_n rec} we have 
\[	l_n = \frac{1}{\sqrt{m^2+4}}\left(\frac{m+\sqrt{m^2+4}}{2}\right)^n - \frac{1}{\sqrt{m^2+4}}\left(\frac{m-\sqrt{m^2+4}}{2}\right)^n
\]
from which it follows that $l_n$ grows with $n$ like $\left(\frac{m+\sqrt{m^2+4}}{2}\right)^n$. The monotonicity of the logarithm now implies
\[	0\leq \frac{\log |F_n|}{l_n} - \frac{\log |A_n|}{l_n} \leq \frac{\log (4^{mn}l_n)}{l_n} \to 0
\]
when $n\to\infty$, which gives that the limit $\lim_{n\to\infty} \frac{\log|A_n|}{l_n}$ exists and equals $\lim_{n\to\infty} \frac{\log|F_n|}{l_n}$.
\end{proof}

\section{Acknowledgement}
The author wishes to thank Michael Baake, Peter Zeiner and Markus Moll at Bielefeld University, Germany, for our discussions of the problem and for reading drafts of the manuscript. This work was supported by the German Research Council (DFG), via CRC 701.

\end{document}